\documentclass{amsart}

\usepackage{amsmath,amssymb,amsthm}

\theoremstyle{plain}
\newtheorem{thm}{Theorem}[section]
\newtheorem{prop}[thm]{Proposition}
\newtheorem{lem}[thm]{Lemma}
\newtheorem{cor}[thm]{Corollary}

\begin{document}

\title[Central ideals and Jacobson radicals of blocks]{Central ideals and Jacobson radicals of blocks of group algebras}
\author[Y. Otokita]{Yoshihiro Otokita}
\address{Yoshihiro Otokita: \newline Department of Mathematics and Informatics \newline Graduate School of Science \newline Chiba University \newline Japan}

\maketitle

\begin{abstract}
We study the relationship between the central ideals and the Jacobson radicals of blocks of group algebras. In particular, we characterize the blocks with the property that the square of the radical is a non-zero central ideal. Moreover, we also consider the blocks with the property that the cube of the radical becomes a non-zero central ideal. These are improvements of a result by K\"{u}lshammer in 2020.  
\end{abstract}

\section{Introduction}
 Throughout this paper we let $G$ be a finite group and $k$ an algebraically closed field of characteristic $p>0$. One of the important problems in modular representation theory is to clarify the Loewy structure of a block $B$ of the group algebra $kG$. Several papers have tried to classify the blocks with the property that the $n$ th power of the Jacobson radical $J(B)$ of $B$ becomes a zero ideal for $n=1, 2, 3, 4$ (see \cite{O} and \cite{KKS}). On the other hand, $B$ has full defect and is isomorphic to the principal block of $kG$ if $B$ has a non-zero central ideal, that is, a non-zero ideal contained in the center of $B$ (see {\cite[Proposition 4.1]{Ku}}), and it is well-known that the structure of the principal block correlates strongly with the property of $G$. In \cite{Ku}, K\"{u}lshammer has proved the following theorem:
 
 \begin{quotation}
\ \textit{$J(B)$ is central, but not a zero ideal if and only if $B$ has a 1-dimensional module and $G$ is a $p$-nilpotent group with non-trivial abelian Sylow $p$-subgroups.}
 \end{quotation}
 
The main purpose of this paper is to improve this result. More precisely, we characterize the blocks with the property that the square of the radical is a non-zero central ideal. Moreover, we determine the blocks with the property that the cube of the radical becomes central when $p$ is odd.

\section{Lemmas}
 In this section, let $A$ be an arbitrary finite-dimensional $k$-algebra with center $Z(A)$ and Jacobson radical $J(A)$. We denote by $L(A)$ the Loewy length of $A$ and we define the commutator subspace of $A$ by
 \[[A, A] := \{ xy-yx \mid x, y \in A\}.\]
Then the following holds.

\begin{lem}\label{Lem1}{\cite[Remark 2.2]{Ku}} $A[A, A]$ is an ideal of $A$.
\end{lem}

Here, we moreover assume that $A$ is a symmetric algebra with symmetrizing linear form $\lambda : A \to k$. For a subspace $X$ of $A$,

\[ X^{\perp} := \{a \in A \mid \lambda(aX)=0 \}\] \

is a subspace of $A$, $(X^{\perp})^{\perp}=X$ and $\dim X + \dim X^{\perp} = \dim A$. If $I$ is an ideal of $A$, then so is $I^{\perp}$. Moreover, it is known that $[A, A]^{\perp}=Z(A)$ and the socle $S(A)$ of $A$ equals $J(A)^{\perp}$. 

 We now prove a lemma about the central ideals of $A$ and the dimensions of $A$-modules.(In this paper, an \textit{$A$-module} always means a finite-generated right $A$-module.) A part of this lemma is due to K\"{u}lshammer {\cite[Remark 2.2 (ii)]{Ku}}. 
 
\begin{lem}\label{Lem2} Let $A$ be a finite-dimensional symmetric $k$-algebra. Then the following are equivalent:
\begin{enumerate}
\item $A$ has a non-zero central ideal, that is, a non-zero ideal contained in $Z(A)$;
\item $A$ has a 1-dimensional module;
\item $A \neq A[A, A]$.
\end{enumerate}
\end{lem}

\begin{proof}
By {\cite[Remark 2.2 (ii)]{Ku}}, (1) implies (2). Suppose (2) holds. Then there exists an algebra homomorphism $\rho : A \to k$ afforded by the 1-dimensional module. Since the kernel $\ker \rho$ of $\rho$ is a proper ideal of $A$ contains $[A, A]$, we have $A[A, A] \subseteq \ker \rho \neq A$ and (3) holds. If (3) holds, then $[A, A] \subseteq A[A, A] \neq A$. Hence $0=A^{\perp}\neq (A[A, A])^{\perp} \subseteq [A, A]^{\perp} =Z(A)$ and we are in (1) by Lemma \ref{Lem1}. 
\end{proof}

 In the following, we consider the central ideals in group algebras. As mentioned in the previous section, let $kG$ be the group algebra of a finite group $G$ over an algebraically closed field $k$ of characteristic $p>0$. Then $kG$ and its blocks are symmetric algebras. In the next lemma, we denote by $G'$ the commutator subgroup of $G$ and by ${X}^{+}$ the sum of all elements in a subset $X$ of $G$.  
 
\begin{lem}\label{Lem3}{\cite[Lemma 2.4]{W}} For an ideal $I$ in a block $B$ of $kG$, the following are equivalent:
\begin{enumerate}
\item $I$ is central;
\item $xg=x$ for all $x \in I$ and $g \in G'$;
\item $I \subseteq B \cdot (G')^{+}$.
\end{enumerate}
\end{lem}

\begin{proof}
This follows from {\cite[Chapter III, Lemma 13.2]{Ka}}.
\end{proof}

 By K\"{u}lshammer {\cite[Proposition 4.1]{Ku}}, a block $B$ of $kG$ has full defect and is isomorphic to the principal block of $kG$ whenever $B$ has a non-zero central ideal. In this case, the property of $G$  as a group is affected by the structure of $B$ as an algebra. In the proposition below, we see the relationship among the Jacobson radical $J(B)$ of $B$, a Sylow $p$-subgroup $P$ of $G$ and the focal subgroup $P^{*}:=P \cap G'$. 

\begin{prop}\label{Prop4} Let $n$ be a positive integer, $B$ a block of $kG$, $P$ a Sylow $p$-subgroup of $G$, and set $P^{*}=P \cap G'$. If $|P^{*}| \neq 1$ and $J(B)^{n}$ is a non-zero central ideal in $B$, then $L(k(P/P^{*})) \le n$.
\end{prop}

\begin{proof}
Put $I:=kG[kG, kG]$. Then $kG/I \simeq k(G/G')$ (see {\cite[Section 4]{Ku}}). Let $A$ be a block of $kG$ and assume $A \neq A[A, A]$. Then, by Lemma \ref{Lem2}, $A$ and $B$ are isomorphic to the principal block of $kG$ and $J(A)^{n}$ is a non-zero central ideal of $A$. Hence $J(A)^{n}$ is contained in $A \cdot (G')^{+}$ (see Lemma \ref{Lem3}). Since the augmentation of $(G')^{+}$ is $|G'|=0$ in $k$, $J(A)^{n} \subseteq I$ (see {\cite[Section 4]{Ku}}) and thus $J(kG)^{n} \subseteq I$. Since $kG/I$ has Jacobson radical $J(kG)+I/I$, this implies $L(k(G/G')) \le n$. Since $G/G'$ is an abelian group with Sylow $p$-subgroup $P/P^{*}$, the claim follows. 
\end{proof}

The Loewy length of $k(P/P^{*})$ is completely determined by the type of $P/P^{*}$. In particular, if $P/P^{*}$ has order $p^{r}$, then $r(p-1)+1 \le L(k(P/P^{*}))$ with equality if and only if $P/P^{*}$ is elementary. 

 We next consider the covering of blocks. 
 
 \begin{lem}\label{Lem5} Let $n$ be a positive integer, $B$  a block of $kG$ and $b$ a block of $kH$ covered by $B$, where $H$ is a normal subgroup of $G$. 
 \begin{enumerate}
 \item If $J(B)^{n}$ is a central ideal of $B$, then $J(b)^{n}$ is a central ideal of $b$.
 \item {\cite[(4.1) Lemma]{KM}} If $|G:H|$ is not divisible by $p$, then $L(B)=L(b)$.
 \end{enumerate}
 \end{lem}
  
\begin{proof}
Remark that $J(kH) \subseteq J(kG)$ since $H$ is normal in $G$. Let $f$ be a primitive idempotent in $b$. Then $fb$ is a principal indecomposable $b$-module. Since $B$ covers $b$, $fb$ is isomorphic to a direct summand of $eB\downarrow_{H}$ for some primitive idempotent $e$ in $B$ and $eBJ(kH)^{n} \subseteq eBJ(kG)^{n} \subseteq J(B)^{n}$. Since we can identify $fJ(b)^{n}=fbJ(kH)^{n}$ with a $kH$-submodule of $eBJ(kH)^{n}$, $G' \cap H$ acts trivially on $fJ(b)^{n}$ by Lemma \ref{Lem3}. Hence we have $J(b)^{n} = \bigoplus fJ(b)^{n} \subseteq Z(b)$ as $H' \subseteq G' \cap H$.
\end{proof}

 We next consider the Loewy structures of principal indecomposable modules. For a $B$-module $U$, we denote by $L(U)$ the Loewy length of $U$. 
 
\begin{prop}\label{Prop6} Let $B$ be a block of $kG$ and $b$ a block of $kG'$ covered by $B$. Moreover, let $U_{0}$ be the principal indecomposable $kG'$-module afforded by the trivial module and $U$ be an arbitrary $b$-module. If $J(B)^{n}$ is a non-zero central ideal of $B$ for some positive integer $n$, then 
\begin{equation*}
L(U) \le \begin{cases}
                  n+1 & (U \simeq U_{0}) \\
                  n & (U \not\simeq U_{0})
             \end{cases}
\end{equation*}
In particular, if $|G:G'|$ is not divisible by $p$, then $L(U_{0})=n+1$.
\end{prop}

\begin{proof}
Set $J=J(b)$. Let $f$ be a primitive idempotent in $b$ such that $fb \simeq U$. By the proof of Lemma \ref{Lem5}, $G'$ acts trivially on $fJ^{n}$. This implies $fJ^{n}=0$, or $fJ^{n}$ is isomorphic to the trivial $kG'$-module. Hence the first claim follows. The second claim is due to Lemma \ref{Lem5} (2) and $n < L(B)$.  
\end{proof}

 At the end of this section, we consider the case that $G$ is a direct product of two groups.

\begin{lem}\label{Lem7} Assume $G=H \times Q$, where $H$ is a subgroup of $G$ and $Q$ is a $p$-subgroup of $G$. Let $B$ be a block of $kG$ and $b$ a block of $kH$ covered by $B$. Moreover, $n$ and $r$ are positive integers such that $n \ge r$. If $J(B)^{n}$ is a non-zero central ideal of $B$ and $J(kQ)^{r}$ is not a zero ideal, then $J(b)^{n-r}$ is central in $b$.
\end{lem}

\begin{proof}
The block idempotents of $B$ and $b$ are the same since $Q$ is a $p$-group. By \cite{L} or \cite{M}, $J(kG)^{n}=\sum_{i=0}^{n}J(kH)^{n-i}J(kQ)^{i} \supseteq J(kH)^{n-r} J(kQ)^{r} \supseteq J(kH)^{n-r} \cdot kQ^{+}$ since $S(kQ)=kQ^{+}$. Hence we obtain $J(b)^{n-r} \cdot kQ^{+} \subseteq J(B)^{n} \subseteq Z(B)$. Let $x \in J(b)^{n-r}$ and $y \in b$. Then $xyQ^{+}=xQ^{+}y=yxQ^{+}$ and $(xy-yx)Q^{+}=0$. Since $G=H \times Q$, this implies $xy-yx=0$. Therefore we deduce $J(b)^{n-r} \subseteq Z(b)$ as claimed.
\end{proof}

\section{Main results}

We prove two main theorems of this paper. In the first theorem below, we remark that it is known that $J(B)^{2}$ is a zero ideal if and only if $|D| \le 2$, where $D$ is a defect group of $B$. 

\begin{thm}\label{Thm8} Let $B$ be a block of $kG$. Then $J(B)^{2}$ is a non-zero central ideal of $B$ if and only if $B$ has a 1-dimensional module and one of the following holds:
\begin{enumerate}
\item $G/O_{p'}(G)$ is an abelian $p$-group of order greater than $2$;
\item $p=2$ and $G/O_{2'}(G)$ is isomorphic to the alternating group of degree $4$;
\item $p=3$ and $G/O_{3'}(G)$ is isomorphic to the symmetric group of degree $3$.
\end{enumerate}
\end{thm}

\begin{proof}
Set $\bar{G}=G/O_{p'}(G)$. If (1), (2) or (3) occurs, then $B \simeq k\bar{G}$. In (1), $0 \neq J(B)^{2} \subseteq B=Z(B)$. In (2) and (3), $L(B)=3$ and all the irreducible $B$-modules have $k$-dimensions 1. Since $\dim S(B) \cap Z(B)$ equals the number of irreducible $B$-modules and $\dim S(B)$ equals the sum of the $k$-dimensions of irreducible $B$-modules, we have $0 \neq J(B)^{2} \subseteq S(B) \subseteq Z(B)$. 

 In the following, put $J=J(B)$ and $S=S(B)$, and let $P$ be a Sylow $p$-subgroup of $G$ to prove the converse. We may assume that $B$ is the principal block of $kG$ by {\cite[Proposition 4.1]{Ku}} and that $G=\bar{G}$ since the canonical map $G \to \bar{G}$ induces an isomorphism between the principal blocks of $kG$ and $k\bar{G}$.  If $p$ does not divide $|G'|$, then $G$ is an abelian $p$-group and (1) holds. Thus we suppose $p$ divides $|G'|$. Let $e$ be a primitive idempotent of $B$. If $eJ^{2}=0$, then $eB$ is irreducible or $eB/eJ \simeq eS=eJ$ since $B$ is symmetric. Hence we have $|P| \le 2$, a contradiction. Thus $eJ^{2} \neq 0$ and $eS \subseteq eJ^{2} \subseteq Z(B)$. From {\cite[Remark 2.2 (ii)]{Ku}}, this implies the $k$-dimension of $eB/eJ \simeq eS$ is 1. Therefore all the irreducible $B$-modules are 1-dimensional and $G' \subseteq P \triangleleft G$ (see the first half of the proof of {\cite[Theorem 4.2]{Ku}}). In particular, $G$ is $p$-solvable and $B = kG$. Since $p$ divides $|G'|$ and $P \triangleleft G$, $J(kP)^{4} \subseteq J^{4} \subseteq (B \cdot (G')^{+})^{2}=0$ and this yields that $P$ is abelian (see {\cite[Proposition 3.1 and Corollary 3.8]{KKS}}) and that $G \neq P$. Hence we can write $P=G' \times C_{P}(H)$ and $G=PH=G'H \times C_{P}(H)$ for some non-trivial $p'$-subgroup $H$ of $G$ by the Schur-Zassenhaus theorem. We now suppose $G' \neq P$. Then $J(kC_{P}(H)) \neq 0$ and it follows from Lemma \ref{Lem7} that $J(k(G'H))$ is a non-zero central ideal in $k(G'H)$. Hence $G'H$ is a $p$-group by {\cite[Corollary 4.5]{Ku}}, but this is a contradiction. Therefore we deduce $G'=P$ and this means $L(kP)=3$ by Proposition \ref{Prop6}. Hence $P$ is a Klein four group and (2) holds, or $P$ is cyclic of order $3$ and (3) holds.  
\end{proof}

The following is a corollary of K\"{u}lshammer {\cite[Theorem 4.3]{Ku}} and Theorem \ref{Thm8}.

\begin{cor}\label{Cor9} Let $B$ be a block of $kG$. 
\begin{enumerate}
\item {\cite[Lemma 2.2]{W}} If $J(B)$ is commutative, then $J(B)^{2}$ is central.
\item $J(B)^{2}$ is central, but $J(B)$ is not commutative if and only if $B$ has a 1-dimensional module and one of the following holds:
\begin{enumerate}
\item $p=2$ and $G/O_{2'}(G)$ is isomorphic to the alternating group of degree $4$;
\item $p=3$ and $G/O_{3'}(G)$ is isomorphic to the symmetric group of degree $3$.
\end{enumerate}
\end{enumerate}
\end{cor}

In the last theorem below, we consider the blocks with the property that the cube of the Jacobson radical becomes a central ideal when $p$ is odd. Remark that the blocks with radical cube zero were classified in Okuyama \cite{O} (see also {\cite[Proposition 3.1]{KKS}}). 

\begin{thm}\label{Thm10} Assume $p \ge 3$ and let $B$ be a block of $kG$.Then $J(B)^{3}$ is a non-zero central ideal of $B$ if and only if $B$ has a 1-dimensional module and $G$ is a $p$-nilpotent group with an abelian Sylow $p$-subgroup of order greater than $4$.
\end{thm}

\begin{proof}
We may assume that $B$ is the principal block of $kG$ by Lemma \ref{Lem2} and {\cite[Proposition 4.1]{Ku}}. Moreover, we may assume $O_{p'}(G)=1$ as mentioned in the proof of Theorem \ref{Thm8}. If $G$ is an abelian $p$-group of order greater that $4$, then $B=kG$ and $0 \neq J(B)^{3} \subseteq B=Z(B)$. In order to prove the converse, assume $0 \neq J(B)^{3} \subseteq Z(B)$. In the following, we suppose $p$ divides $|G'|$ and deduce a contradiction. By {\cite[Theorem 1.2]{Ko}} and Proposition \ref{Prop6}, $|G:G'|$ is divisible by $p$. Hence we have from Proposition \ref{Prop4} that $p=3$ and $|G:G'|_{3}=3$. Moreover, $|G'|_{3}=3$ by \cite{O} (see again {\cite[Theorem 1.2]{Ko}} and Proposition \ref{Prop6}). Hence $G$ has an abelian Sylow $3$-subgroup $P$ of order $9$. Let $b$ be the principal block of $kO^{3'}(G)$. From Lemma \ref{Lem5}, $J(b)^{3}$ is a non-zero central ideal in $b$. By Fong \cite{F} (see {\cite[Proposition 4.3]{KY}}), we can write $O^{3'}(G)=L \times Q$, where $L$ is a direct product of some simple groups such that $L=L'$ and $Q$ is a $3$-group. By {\cite[Theorem 1.2]{Ko}} and Proposition \ref{Prop6}, we have $|Q| \neq 1$. Hence $J(kQ)^{2} \neq 0$ and this implies $|L|=1$ by Lemma \ref{Lem7} and {\cite[Corollary 4.5]{Ku}}. Therefore $P$ is normal in $G$ and $B = kG$. By the Schur-Zassenhaus theorem, we can write $P=P^{*} \times C_{P}(H)$ and $G=PH=P^{*}H \times C_{P}(H)$ for some $p'$-subgroup $H$ of $G$, where $P^{*}:=P \cap G'$.As proved above, $P^{*} \neq P$. Hence $J(kC_{P}(H))^{2} \neq 0$ and it follows from Lemma \ref{Lem7} that $J(k(P^{*}H))$ is a non-zero central ideal. Thus $|H|=1$ by {\cite[Corollary 4.5]{Ku}}, but this is a contradiction. Therefore we deduce that $|G'|$ is not divisible by $p$, $G$ is an abelian $p$-group and the claim follows. 
\end{proof}

 If $p=2$ and $J(B)^{3}$ is a non-zero central ideal, then we can give an upper bound of the order of a Sylow $2$-subgroup of $G$ by the same way above, but can not determine the structure of $G$. We don't know if $G$ is $2$-solvable.


\begin{thebibliography}{99}

\bibitem{F}
P. Fong,
Manuscript, 1996.

\bibitem{Ka}
G. Karpilovsky,
The Jacobson Radical of Group Algebras,
North-Holland Publishing Co., Amsterdam (1987).

\bibitem{Ko}
S. Koshitani,
\textit{The projective cover of the trivial module over a group algebra of a finite group},
Comm. Algebra \textbf{42} (2014), 4308--4321.

\bibitem{KKS}
S. Koshitani, B. K\"{u}lshammer and B. Sambale,
\textit{On Loewy lengths of blocks},
Math. Proc. Cambridge Philos. Soc. \textbf{156} (2014), 555--570.

\bibitem{KM}
S. Koshitani and H. Miyachi,
\textit{Donovan conjecture and Loewy length for principal 3-blocks of finite groups with elementary abelian Sylow 3-subgroups of order 9},
Comm. Algebra \textbf{29} (2001), 4509--4522.

\bibitem{KY}
S. Koshitani and Y. Yoshii,
\textit{Eigenvalues of Cartan matrices of principal 3-blocks of finite groups with abelian Sylow 3-subgroups},
J. Algebra \textbf{324} (2010), 1985--1993.

\bibitem{Ku}
B. K\"{u}lshammer,
\textit{Centers and radicals of group algebras and blocks},
Arch. Math. (Basel) \textbf{114} (2020), 619--629.

\bibitem{L}
C. Loncour,
\textit{Radical d'une alg\`{e}bre d'un produit direct de groupes finis.},
Bull. Soc. Math. Belg. \textbf{23} (1971), 423--435.

\bibitem{M}
K. Motose,
\textit{On C. Loncour's results},
Proc. Japan Acad. \textbf{50} (1974), 570-571.

\bibitem{O}
T. Okuyama,
\textit{On blocks of finite groups with radical cube zero},
Osaka J. Math. \textbf{23} (1986), 461--465.

\bibitem{W}
D.A.R. Wallace,
\textit{On commutative and central conditions on the Jacobson radical of the group algebra of a group},
Proc. London Math. Soc. (3) \textbf{19} (1969), 385--402.

\end{thebibliography}
\end{document}